\documentclass[a4paper,11pt]{article}

\usepackage[top=3.0cm, bottom=3.0cm, inner=3.0cm, outer=3.0cm, includefoot]{geometry}

\usepackage{verbatim}

\usepackage{geometry}
\usepackage{amssymb}
\usepackage{amsmath}
\usepackage{graphicx}
\usepackage{amsthm}

\usepackage{color}

\usepackage[T1]{fontenc}
\usepackage[utf8]{inputenc}
\usepackage{authblk}

\usepackage{enumerate}
\usepackage{tikz}
\usetikzlibrary{arrows}
\usetikzlibrary{decorations.markings}

\newcommand{\eChar}{\begin{enumerate}[(i)]}
\newcommand{\eCharR}{\begin{enumerate}[(a)]}
\newcommand{\eBr}{\begin{enumerate}[(1)]}

\newcommand\floor[1]{\left\lfloor #1\right\rfloor}
\newcommand\ceil[1]{\left\lceil #1\right\rceil}

\title
{Ollivier-Ricci idleness functions of graphs}

\author[1]{D. P. Bourne}
\author[1]{D. Cushing}
\author[2]{S. Liu}
\author[3]{F. M\"unch}
\author[1]{N. Peyerimhoff}
\affil[1]{Department of Mathematical Sciences, Durham University}
\affil[2]{School of Mathematical Sciences, University of Science and Technology of China}
\affil[3]{Institute of Mathematics, Universit\"at Potsdam }
\date{\today}

\theoremstyle{plain}
\newtheorem{lemma}{Lemma}[section]
\newtheorem{theorem}[lemma]{Theorem}

\newtheorem{corollary}[lemma]{Corollary}

\theoremstyle{definition}

\newtheorem{definition}{Definition}[section]

\newtheorem{rem}[lemma]{Remark}

\numberwithin{equation}{section}

\begin{document}

\maketitle

\begin{abstract}
  We study the Ollivier-Ricci curvature of graphs as a function of the
  chosen idleness. We show that this idleness function is concave and
  piecewise linear with at most $3$ linear parts, with at most $2$
  linear parts in the case of a regular graph. We then apply our
  result to show that the idleness function of the Cartesian product
  of two regular graphs is completely determined by the idleness
  functions of the factors.
\end{abstract}

\section{Introduction and statement of results}
Ricci curvature plays a very important role in the study of Riemannian manifolds. In the discrete setting of graphs, there is very active recent research on various types of Ricci curvature notions and their applications.

In \cite{Oll} Ollivier developed a notion of Ricci curvature of Markov
chains valid on metric spaces including graphs. In this notion an {\it idleness} parameter, $p\in [0,1],$ must be set in order to obtain a curvature $\kappa_{p}$. Ollivier considered idleness $0$ and $\frac{1}{2}.$ For graphs, Ollivier's notion for idleness 0 has been studied further in \cite{BM15, Im, JL14, LY10, P}. In \cite{OllVil} Ollivier and Villani considered idleness $\frac{1}{d+1},$ where $d$ is the degree of a regular graph, in order to investigate the curvature of the hypercube. In \cite{LLY11}, Lin, Lu, and Yau modified the definition of Ollivier-Ricci curvature to compute the derivative of the curvature with respect to the idleness, which they denote by $\kappa$.

We will show that for a regular graph the following holds:
$$\kappa = 2\kappa_{\frac{1}{2}} = {\small \frac{d+1}{d}}\kappa_{\frac{1}{d+1}}.$$
Therefore some of these different curvature notions are related to
each other by scaling factors.

In \cite{BM15}, Bhattacharya and Mukherjee derive exact expressions of Ollivier-Ricci curvature for bipartite graphs in the special case of idleness $p=0$
and for graphs of girth at least $5$. They use this result to classify all
graphs with $\kappa_{0} = 0$ for all edges (called `Ricci flat' in
their paper) and girth at least $5$. There is a small overlap between some
of our methods and theirs in this paper (for example they discuss the
existence of integer-valued optimal Kantorovich potentials in
the special case of vanishing idleness ($p = 0$)).

To our knowledge, the \emph{global} piecewise linear structure of the
function $p \mapsto \kappa_{p}$ has not yet been established and
the only concrete examples in the literature, where the full idleness function is
computed are the hypercube and the
complete graphs; see \cite{LLY11}. However, some properties of this
function have been discussed. In \cite{LLY11}, it was shown that the
idleness function is concave. It was shown in \cite{LoRo} that
$\kappa_{p}$ is linear close to idleness $p=1$, and in \cite{Smith14}
it was shown that $\kappa_{p}$ is linear if a certain condition is
satisfied (see the introductory part of Section \ref{first}).

Throughout this article, let $G=(V,E)$ be a locally finite graph with
vertex set $V$, edge set $E$, and which contains no multiple edges or
self loops. Let $d_x$ denote the degree of the vertex $x\in V$ and 
$d(x,y)$ denote the length of the shortest path between two vertices
$x$ and $y$, that is, the combinatorial distance. We denote the existence of an edge between $x$ and $y$ by $x\sim y.$

We define the following probability measures $\mu_x$ for any
$x\in V,\: p\in[0,1]$:
$$\mu_x^p(z)=\begin{cases}p,&\text{if $z = x$,}\\
\frac{1-p}{d_x},&\text{if $z\sim x$,}\\
0,& \mbox{otherwise.}\end{cases}$$
Let $W_{1}$ denote the 1-Wasserstein distance between two probability measures on $V$, see \cite{Vill03} page 211.  The $ p-$Ollivier-Ricci curvature of an edge $x\sim y$ in $G=(V,E)$ is
$$\kappa_{ p}(x,y)=1-W_1(\mu^{ p}_x,\mu^{ p}_y).$$
Y. Lin, L. Lu, and S.T. Yau introduced in \cite{LLY11} the following
Ollivier-Ricci curvature:
$$\kappa(x,y) = \lim_{ p\rightarrow 1}\frac{\kappa_{ p}(x,y)}{1- p}.$$
Note that their curvature notion does not have an idleness index, which distinguishes their notion from the idleness function  $p \mapsto \kappa_{p}(x,y)$ in this paper, which we call the {\it Ollivier-Ricci idleness function}.
We will show that $$ \kappa_p(x,y) = (1-p)\kappa(x,y) $$for all $p \in \left [\frac{1}{\max\{d_x,d_y\}+1},1 \right]$ and that $\kappa_{0}(x,y)\leq \kappa(x,y)\leq \kappa_{0}(x,y)+\frac{2}{\max\{d_x,d_y\}}.$ Observe that $\kappa(x,y) = -\kappa'_{1}(x,y).$

Next, we give some examples of graphs and their Ollivier-Ricci idleness function at a particular edge $x \sim y$.
\\
\\
{\bf Examples:$\quad$} Below is the one-path and a plot of the corresponding idleness function:
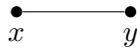
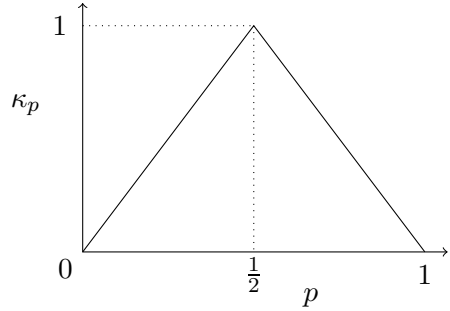
\begin{center}
\begin{tikzpicture}[x=1.5cm, y=1.5cm,
	vertex/.style={
		shape=circle, fill=black, inner sep=1.5pt	
	}
]

\node[vertex, label=below:$x$] (1) at (0, 1) {};
\node[vertex, label=below:$y$] (2) at (1, 1) {};

\draw (1) -- (2);

\draw [<->] (3,2.2) -- (3,0) -- (6.2,0);
\draw (3,0) -- (4.5, 2) -- (6, 0);

\draw[dotted] (4.5, 2) -- (4.5, 0);
\draw[dotted] (3, 2) -- (4.5, 2);

\node at (4.5,-0.2) {$\frac{1}{2}$};
\node at (6,-0.2) {$1$};
\node at (2.8, 2) {$1$};
\node at (2.85, -0.15) {$0$};
\node at (2.5, 1.3) {$\kappa_{p}$};
\node at (5,-0.4) {$p$};


\end{tikzpicture}
\end{center}

We now present the idleness function for $3-,4-$ and $5-$ cycles:

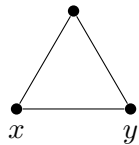
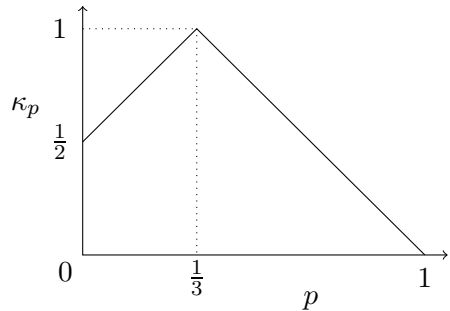
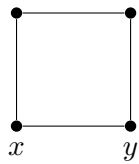
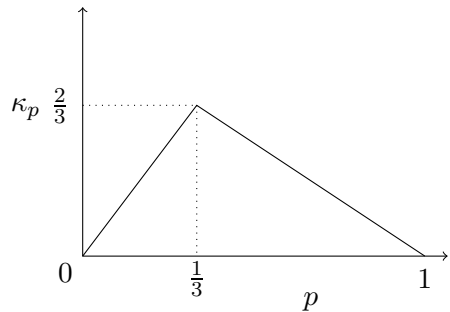
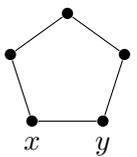
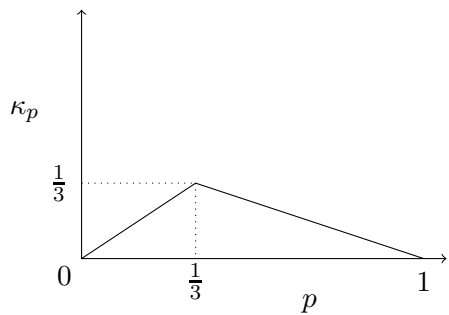
\begin{center}
\begin{tikzpicture}[x=1.5cm, y=1.5cm,
	vertex/.style={
		shape=circle, fill=black, inner sep=1.5pt	
	}
]

\node[vertex, label=below:$x$] (1) at (0, 0) {};
\node[vertex] (3) at (0.5, 0.866) {};
\node[vertex, label=below:$y$] (2) at (1, 0) {};

\draw (1) -- (2);
\draw (1) -- (3);
\draw (3) -- (2);

\draw [<->] (3,2.2) -- (3,0) -- (6.2,0);
\draw (3,1) -- (4, 2) -- (6, 0);

\draw[dotted] (4, 2) -- (4, 0);
\draw[dotted] (3, 2) -- (4, 2);

\node at (4,-0.2) {$\frac{1}{3}$};
\node at (6,-0.2) {$1$};
\node at (2.8, 2) {$1$};
\node at (2.8, 1) {$\frac{1}{2}$};
\node at (2.85, -0.15) {$0$};
\node at (2.5, 1.3) {$\kappa_{p}$};
\node at (5,-0.4) {$p$};

\end{tikzpicture}
\end{center}

\begin{center}
\begin{tikzpicture}[x=1.5cm, y=1.5cm,
	vertex/.style={
		shape=circle, fill=black, inner sep=1.5pt	
	}
]

\node[vertex, label=below:$x$] (bot1) at (0, 0) {};
\node[vertex] (bot2) at (1, 1) {};
\node[vertex] (top1) at (0, 1) {};
\node[vertex, label=below:$y$] (top2) at (1, 0) {};

\foreach \i in {1, 2}{
	\foreach \j in {1, 2}{
		\draw (bot\i) -- (top\j);
	}
}

\draw [<->] (3,2.2) -- (3,0) -- (6.2,0);
\draw (3,0) -- (4, 1.333) -- (6, 0);

\draw[dotted] (4, 1.333) -- (4, 0);
\draw[dotted] (3, 1.333) -- (4, 1.333);

\node at (4,-0.2) {$\frac{1}{3}$};
\node at (6,-0.2) {$1$};
\node at (2.8, 1.333) {$\frac{2}{3}$};
\node at (2.85, -0.15) {$0$};
\node at (2.5, 1.3) {$\kappa_{p}$};
\node at (5,-0.4) {$p$};

\end{tikzpicture}
\end{center}

\begin{center}
\begin{tikzpicture}[x=1.5cm, y=1.5cm,
	vertex/.style={
		shape=circle, fill=black, inner sep=1.5pt	
	}
]

\node[vertex] (1) at (0.5, 1.363) {};
\node[vertex] (2) at (0, 1) {};
\node[vertex, label=below:$x$] (3) at (0.189, 0.411) {};
\node[vertex, label=below:$y$] (4) at (0.811, 0.411) {};
\node[vertex] (5) at (1, 1) {};

\draw (1) -- (2);
\draw (2) -- (3);
\draw (3) -- (4);
\draw (4) -- (5);
\draw (5) -- (1);

\draw [<->] (3,2.2) -- (3,0) -- (6.2,0);
\draw (3,0) -- (4, 0.667) -- (6, 0);

\draw[dotted] (4, 0.667) -- (4, 0);
\draw[dotted] (3, 0.667) -- (4, 0.667);

\node at (4,-0.2) {$\frac{1}{3}$};
\node at (6,-0.2) {$1$};
\node at (2.8, 0.667) {$\frac{1}{3}$};
\node at (2.85, -0.15) {$0$};
\node at (2.5, 1.3) {$\kappa_{p}$};
\node at (5,-0.4) {$p$};

\end{tikzpicture}
\end{center}
For cycles of length $6$ or greater the idleness function at every
edge vanishes identically (we call those edges \emph{bone idle}; see Section 7).

So far we have only seen idleness functions with at most $2$ linear
parts. We will show that if $d_{x} = d_{y},$ then this is always the
case. However, if $d_{x}\neq d_{y}$, then $3$ linear parts may
occur, as shown in the following example:



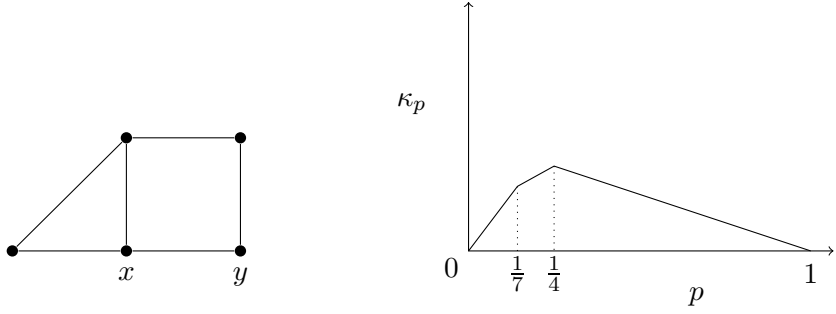
\begin{center}
\begin{tikzpicture}[x=1.5cm, y=1.5cm,
	vertex/.style={
		shape=circle, fill=black, inner sep=1.5pt	
	}
]

\node[vertex, label=below:$x$] (1) at (0, 0) {};
\node[vertex] (3) at (1, 1) {};
\node[vertex] (4) at (0, 1) {};
\node[vertex, label=below:$y$] (2) at (1, 0) {};
\node[vertex] (5) at (-1, 0) {};

\draw (1) -- (2);
\draw (1) -- (5);
\draw (3) -- (4);
\draw (1) -- (4);
\draw (2) -- (3);
\draw (5) -- (4);

\draw [<->] (3,2.2) -- (3,0) -- (6.2,0);
\draw (3,0) -- (3.429, 0.571) -- (3.75, 0.75) -- (6, 0);

\draw[dotted] (3.429, 0.571) -- (3.429, 0);
\draw[dotted] (3.75, 0.75) -- (3.75, 0);

\node at (3.429,-0.2) {$\frac{1}{7}$};
\node at (3.75,-0.2) {$\frac{1}{4}$};
\node at (6,-0.2) {$1$};
\node at (2.85, -0.15) {$0$};
\node at (2.5, 1.3) {$\kappa_{p}$};
\node at (5,-0.4) {$p$};

\end{tikzpicture}
\end{center}





In fact the Ollivier-Ricci idleness function is piecewise linear with
\emph{at most} 3 parts always, a fundamental fact which is included in
the following theorem (our main result):

\begin{theorem}\label{main}
Let $G=(V,E)$ be a locally finite graph. Let $x,y\in V$ with $x\sim y.$ Then the function $p \mapsto \kappa_{p}(x,y)$ is concave and piecewise linear over $[0,1]$ with at most $3$ linear parts. Furthermore $\kappa_{p}(x,y)$ is linear on the intervals
\begin{equation*}
\left[0,\frac{1}{{\rm{lcm}}(d_{x},d_{y})+1}\right]\:\:\: {\rm and} \:\:\:\left[\frac{1}{\max(d_{x},d_{y})+1},1\right].
\end{equation*}
Thus, if we have the further condition $d_{x}=d_{y}$, then $\kappa_{p}(x,y)$ has at most two linear parts.
\end{theorem}

In our above example of $3$ linear parts the changes in slope occurs at $\frac{1}{{\rm{lcm}}(d_{x},d_{y})+1}$ and $\frac{1}{\max(d_{x},d_{y})+1}.$ However this need not always be the case. Consider the following example:
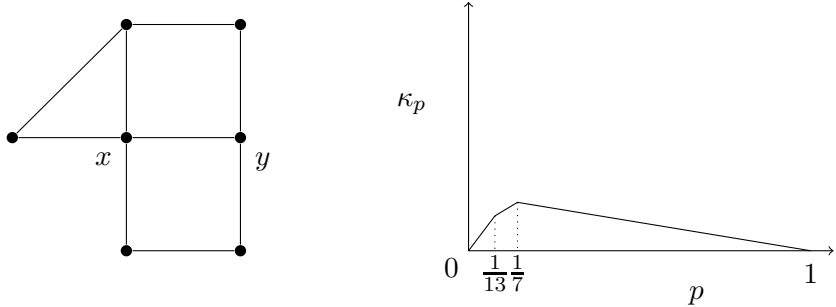
\begin{center}
\begin{tikzpicture}[x=1.5cm, y=1.5cm,
	vertex/.style={
		shape=circle, fill=black, inner sep=1.5pt	
	}
]

\node[vertex, label=below left:$x$] (1) at (0, 1) {};
\node[vertex] (3) at (1, 2) {};
\node[vertex] (4) at (0, 2) {};
\node[vertex, label=below right:$y$] (2) at (1, 1) {};
\node[vertex] (5) at (-1, 1) {};
\node[vertex] (6) at (0, 0) {};
\node[vertex] (7) at (1, 0) {};

\draw (1) -- (2);
\draw (1) -- (5);
\draw (3) -- (4);
\draw (1) -- (4);
\draw (2) -- (3);
\draw (5) -- (4);
\draw (6) -- (7);
\draw (1) -- (6);
\draw (2) -- (7);

\draw [<->] (3,2.2) -- (3,0) -- (6.2,0);
\draw (3,0) -- (3.231, 0.308) -- (3.429, 0.428) -- (6, 0);

\draw[dotted] (3.231, 0.308) -- (3.231, 0);
\draw[dotted] (3.429, 0.428) -- (3.429, 0);

\node at (3.231,-0.2) {$\frac{1}{13}$};
\node at (3.429,-0.2) {$\frac{1}{7}$};
\node at (6,-0.2) {$1$};
\node at (2.85, -0.15) {$0$};
\node at (2.5, 1.3) {$\kappa_{p}$};
\node at (5,-0.4) {$p$};
\end{tikzpicture}
\end{center}
Here the first change in gradient did occur at
$\frac{1}{{\rm{lcm}}(d_{x},d_{y})+1}=\frac{1}{13}$, but the second change in
gradient occurs before $\frac{1}{\max(d_{x},d_{y})+1}=\frac{1}{5}.$

\begin{rem}\label{concavity}
Since $\kappa_{p}(x,y) = 1 - W_{1}(\mu_{x}^{p},\mu_{y}^{p})$ and $ W_{1}(\mu_{x}^{p},\mu_{y}^{p})$ is the supremum of affine functions of $p$ (by the Kantorovich Duality Theorem), then $p\mapsto  W_{1}(\mu_{x}^{p},\mu_{y}^{p})$ is convex and so $p\mapsto \kappa_{p}(x,y)$ is concave. An alternative proof of concavity was given in \cite{LLY11}.
\end{rem}

A consequence of Theorem \ref{main} and the results in \cite{LLY11} is
the following Corollary.

\begin{corollary}\label{cartcor}
Let $G=(V_{G},E_{G})$ be a $d_{G}$-regular graph and $H=(V_{H},E_{H})$ be a $d_{H}$-regular graph. Let $x_{1},x_{2}\in V_{G}$ with $x_{1}\sim x_{2}$ and $y\in V_{H}$. Then
\begin{align*}
& \kappa^{G\times H}_{ p}((x_{1},y),(x_{2},y))
\\
& = \begin{cases} \frac{d_{G}}{d_{G}+d_{H}}\kappa^{G}_{ p}(x_{1},x_{2})+\frac{d_{G}d_{H}}{d_{G}+d_{H}}(\kappa^{G}(x_{1},x_{2})-\kappa^{G}_{0}(x_{1},x_{2})) p,&  \text{if $p \in [0,\frac{1}{d_{G}+d_{H}+1}]$,}\\
\frac{d_{G}}{d_{G}+d_{H}} \kappa^{G}(x_{1},x_{2})(1- p), &  \text{if $p\in [\frac{1}{d_{G}+d_{H}+1},1]$.}
\end{cases}
\end{align*}
\end{corollary}

This result shows that the idleness function of the Cartesian product of two regular graphs is completely determined by the idleness functions of the factors.

We finish this introduction with an outline of the rest of this paper. In Section \ref{sect:defn-nots} we present the relevant notation and background material.
In Section \ref{LinSection} we show that $p \mapsto \kappa_{p}$ is piecewise linear with at most $3$ linear parts. In Sections \ref{last} and \ref{first} we give bounds on the size of the first and last linear part.
We prove Corollary \ref{cartcor} in Section \ref{CartSection}. Finally, in Section \ref{QSection}, we present some open questions. Moreover we discuss the problem of characterising edges with globally linear curvature functions.

\section{Definitions and notation}\label{sect:defn-nots}
We now introduce the relevant definitions and notation we will need in this paper. First, we recall the Wasserstein distance and
the Ollivier-Ricci curvature.

\begin{definition}
Let $G = (V,E)$ be a locally finite graph. Let $\mu_{1},\mu_{2}$ be two probability measures on $V$. The {\it Wasserstein distance} $W_1(\mu_{1},\mu_{2})$ between $\mu_{1}$ and $\mu_{2}$ is defined as
\begin{equation} \label{eq:W1def}
W_1(\mu_{1},\mu_{2})=\inf_{\pi \in \Pi(\mu_1,\mu_2)} \sum_{y\in V}\sum_{x\in V} d(x,y)\pi(x,y),
\end{equation}
where 
\[
\Pi(\mu_1,\mu_2) = \left\{ \pi: V \times V \to [0,1] : \mu_{1}(x)=\sum_{y\in V}\pi(x,y), \;
\mu_{2}(y)=\sum_{x\in V}\pi(x,y) 
\right\}.
\]
\end{definition}

The transportation plan $\pi$ moves a mass
distribution given by $\mu_1$ into a mass distribution given by
$\mu_2$, and $W_1(\mu_1,\mu_2)$ is a measure for the minimal effort
which is required for such a transition.
If $\pi$ attains the infimum in \eqref{eq:W1def} we call it an {\it
  optimal transport plan} transporting $\mu_{1}$ to $\mu_{2}$.

\begin{definition}
The $ p-$Ollivier-Ricci curvature of an edge $x\sim y$ in $G=(V,E)$ is
$$\kappa_{ p}(x,y)=1-W_1(\mu^{ p}_x,\mu^{ p}_y),$$
where $p$ is called the {\it idleness}.
\end{definition}

A fundamental concept in optimal transport theory and vital to our work is Kantorovich duality. First we recall the notion
of 1--Lipschitz functions and then state the Kantorovich Duality Theorem.

\begin{definition}
Let $G=(V,E)$ be a locally finite graph, $\phi:V\rightarrow\mathbb{R}.$ We say that $\phi$ is $1$-Lipschitz if
$$|\phi(x) - \phi(y)| \leq d(x,y)$$
for all $x,y\in V.$ Let \textrm{1--Lip} denote the set of all $1$--Lipschitz functions on $V$.
\end{definition}

\begin{theorem}[Kantorovich duality \cite{Vill03}]\label{Kantorovich}
Let $G = (V,E)$ be a locally finite graph. Let $\mu_{1},\mu_{2}$ be two probability measures on $V$. Then
$$W_1(\mu_{1},\mu_{2})=\sup_{\substack{\phi:V\rightarrow \mathbb{R}\\ \phi\in \textrm{\rm{1}--{\rm Lip}}}}  \sum_{x\in V}\phi(x)(\mu_{1}(x)-\mu_{2}(x)).$$
If $\phi \in \textrm{\rm{1}--{\rm Lip}}$ attains the supremum we call it an \emph{optimal Kantorovich potential} transporting $\mu_{1}$ to $\mu_{2}$.
\end{theorem}

\section{Properties of the idleness function}\label{LinSection}
In this section we prove that the Ollivier-Ricci idleness function has at most $3$ linear parts. Two ingredients of this proof are the `integer-valuedness' of optimal Kantorovich potentials and the Complementary Slackness Theorem, which we state and prove now.
\begin{lemma}\label{mass-sharpness}
Let $G= (V,E)$ be a locally finite graph. Let $x,y\in V$ with $x\sim y.$ Let $p\in [0,1].$ Let $\pi$ and $\phi$ be an optimal transport plan and an optimal Kantorovich potential transporting $\mu^{p}_{x}$ to $\mu^{p}_{y},$ respectively. Let $u,v\in V$ with $\pi(u,v)\neq 0.$ Then
$$\phi(u) - \phi(v) = d(u,v).$$
\end{lemma}
This follows from the Complementary Slackness Theorem (see, for example, \cite[page 49]{Min}) or from standard results in optimal transport theory \cite[page 88]{Vill03}. For the sake of completeness we include a short proof here.
\begin{proof}
By the definitions of $\phi$ and $\pi$ we have
$$W_{1}(\mu^{p}_{x}, \mu^{p}_{y})  = \sum_{w\in V} \phi(w)(\mu^{p}_{x}-\mu^{p}_{y})(w) = \sum_{w\in V} \sum_{z\in V} d(w,z) \pi(w,z),$$
and
$$\sum_{w\in V}\pi(w,z) = \mu^{p}_{y}(z),\:\:\:\:\: \sum_{z\in V}\pi(w,z)  = \mu^{p}_{x}(w).$$
Then
\begin{align*}
W_{1}(\mu^{p}_{x}, \mu^{p}_{y}) &  = \sum_{w\in V} \phi(w)\mu^{p}_{x}(w)-\sum_{z\in V} \phi(z)\mu^{p}_{y}(z)
\\
& = \sum_{w\in V} \phi(w)\sum_{z\in V}\pi(w,z)-\sum_{z\in V} \phi(z)\sum_{w\in V}\pi(w,z)
\\
& = \sum_{w\in V} \sum_{z\in V} (\phi(w)-\phi(z)) \pi(w,z)
\\
& \leq \sum_{w\in V} \sum_{z\in V} d(w,z) \pi(w,z)
\\
& = W_{1}(\mu^{p}_{x}, \mu^{p}_{y}).
\end{align*}
Thus
$$\sum_{w\in V} \sum_{z\in V} (\phi(w)-\phi(z)) \pi(w,z) = \sum_{w\in V} \sum_{z\in V} d(w,z) \pi(w,z).$$
Therefore
$$\phi(w)-\phi(z)<d(w,z)\implies \pi(w,z) = 0,$$
thus completing the proof.
\end{proof}

As mentioned in the introduction, in \cite{BM15} the authors discuss the existence of integer-valued optimal Kantorovich potentials in the special case of vanishing idleness ($p = 0$). We first introduce the floor and ceiling of functions and then state a corresponding result for the case of arbitrary idleness.

\begin{definition}
Let $G=(V,E)$ be a locally finite graph and let $\phi :V \rightarrow \mathbb{R}.$ Define the functions $\floor{\phi}$ and $\ceil{\phi}$ as follows:
\begin{align*}
\floor{\phi}: V & \rightarrow \mathbb{R}
\\
v & \mapsto \floor{\phi(v)},
\\
\ceil{\phi}: V & \rightarrow \mathbb{R}
\\
v & \mapsto \ceil{\phi(v)}.
\end{align*}
\end{definition}

\begin{lemma}\label{fandc}
Let $G=(V,E)$ be a locally finite graph. Let $\phi\in \textrm{\rm{1}--{\rm Lip}}.$ Then $\floor{\phi}, \ceil{\phi}\in \textrm{\rm{1}--{\rm Lip}}.$
\end{lemma}

\begin{proof}
For each $v\in V$ set $\delta_{v} = \phi(v) - \floor{\phi(v)}.$ Note that $\delta_{v}\in[0,1).$ Then
$$|\floor{\phi(v)}-\floor{\phi(w)}|  = |\phi(v)-\delta_{v} -\phi(w) +\delta_{w}|\leq d(v,w) + |\delta_{v}- \delta_{w}|.$$
Since $\delta_{v}- \delta_{w}\in (-1,1)$ we have $|\floor{\phi(v)}-\floor{\phi(w)}| < d(v,w) +1 $ and so $|\floor{\phi(v)}-\floor{\phi(w)}| \leq d(v,w)$ since $|\floor{\phi(v)}-\floor{\phi(w)}|$ is integer valued. Thus $\floor{\phi}\in \textrm{1--Lip}.$ The proof that $\ceil{\phi}\in \textrm{1--Lip}$ follows similarly.
\end{proof}

\begin{lemma}[Integer-Valuedness]\label{integerness}
Let $G=(V,E)$ be a locally finite graph. Let $x,y\in V$ with $x\sim y.$ Let $ p\in [0,1].$ Then there exists $\phi\in \textrm{\rm{1}--{\rm Lip}}$ such that
$$W_1(\mu^{ p}_{x},\mu^{ p}_{y})= \sum_{w\in V}\phi(w)(\mu^{ p}_{x}(w)-\mu^{ p}_{y}(w)),$$
and $\phi(w)\in\mathbb{Z}$ for all $w\in V$.
\end{lemma}

\begin{proof}
Let $\Phi$ be an optimal Kantorovich potential transporting $\mu^{p}_{x}$ to $\mu^{p}_{y}.$ Let $\pi$ be an optimal transport plan transporting $\mu^{p}_{x}$ to $\mu^{p}_{y}.$ Consider the following graph $H$ with vertices $V$ and edges given by the following adjacency matrix $A$:
\begin{gather*}
A(v,w)=1 \:\:\:\textrm{if}\:\:\: \pi(v,w)=1\:\:\: \textrm{or}\:\:\: \pi(w,v)=1,
\\
A(v,w)=0 \:\:\:\textrm{otherwise}.
\end{gather*}
Let $(W_{i})_{i=1}^{n}$ denote the connected components of $H$. Fix $u,v\in W_{i}$ for some $i\in\{1,\ldots, n\}.$ By Lemma \ref{mass-sharpness} we have $|\Phi(u)-\Phi(v)| = d(u,v).$

Define $\phi:V\rightarrow \mathbb{R}$ as follows
$$\phi(v) = \sup \{ \psi(v) : \psi:V \to \mathbb{Z},\:\: \psi\in \textrm{1--Lip}, \, \psi \le \Phi \}.$$
By definition, $\phi$ is an integer-valued 1--Lipschitz function and $\phi\leq\Phi.$ Note that $\phi=\floor{\Phi}$ since $\floor{\Phi}\in \textrm{1--Lip}$ by Lemma \ref{fandc}.

Finally we must show that $\phi$ is optimal.
For each $v\in V$ set $\delta_{v} = \Phi(v) - \floor{\Phi(v)} = \Phi(v) - \phi(v)$. Note that $\mu_{x}^{p}(W_{i}) = \mu_{y}^{p}(W_{i})$ for all $i$ (since no mass is transported between different connected components $W_{i}$), and that $\delta_{u} =\delta_{v}$ if $u,v$ belong to the same component $W_{i}$, for some $i$. Set $\delta_{i}=\delta_{u}$ for any $u\in W_{i}.$ Then
\begin{align*}
\sum_{w\in V}\phi(w)(\mu^{ p}_{x}(w)-\mu^{ p}_{y}(w)) = & \sum_{w\in V}(\Phi(w)-\delta_{w})(\mu^{ p}_{x}(w)-\mu^{ p}_{y}(w))
\\
 = & \sum_{i=1}^{n}\sum_{w\in W_{i}}(\Phi(w)-\delta_{w})(\mu^{ p}_{x}(w)-\mu^{ p}_{y}(w))
\\
 = & \sum_{i=1}^{n}\sum_{w\in W_{i}}\Phi(w)(\mu^{ p}_{x}(w)-\mu^{ p}_{y}(w))
\\
& -\sum_{i=1}^{n}\sum_{w\in W_{i}}\delta_{w}(\mu^{ p}_{x}(w)-\mu^{ p}_{y}(w))
\\
= & \sum_{w\in V}\Phi(w)(\mu^{ p}_{x}(w)-\mu^{ p}_{y}(w)) - \sum_{i=1}^{n}\delta_{i}\underbrace{(\mu^{ p}_{x}(W_{i})-\mu^{ p}_{y}(W_{i})}_{=0})
\\
= & W_{1}(\mu_{x}^{p},\mu_{y}^{p}).
\end{align*}
Therefore $\phi$ is optimal, as required.
\end{proof}

Now we formulate our main result of this section.

\begin{theorem}\label{linear}
Let $G=(V,E)$ be a locally finite graph. Let $x,y\in V$ with $x\sim y.$ Then $p \mapsto \kappa_{p}(x,y)$ is piecewise linear over $[0,1]$ with at most $3$ linear parts.
\end{theorem}

\begin{proof}
For $\phi:V\rightarrow \mathbb{R},$ let
\begin{equation}
\label{eq:F}
F(\phi) = d_{y}\left(\sum_{\substack{z\sim x\\ z\neq y}}\phi(z)\right)-d_{x}\left(\sum_{\substack{z\sim y\\ z\neq x}}\phi(z)\right).
\end{equation}
For $j\in\{-1,0,1\}$, define 
\begin{equation} \label{eq:optprobsep2}
\mathcal{A}_{j} = \{\phi:V\rightarrow\mathbb{Z}: \phi(x)=j, \, \phi(y)=0, \, \phi \in \textrm{1--Lip}\}
\end{equation}
and define the constants
$$c_{j} = \sup_{\phi\in\mathcal{A}_{j}} F(\phi).$$
Finally we define the linear maps
$$f_{j}( p) =  \left(p  -\frac{1- p}{d_y}\right)j+\frac{1- p}{d_{x}d_{y}}c_{j}.$$
Then
\begin{align}
\nonumber
W_1(\mu^{ p}_{x},\mu^{ p}_{y}) & = \sup_{\phi\in \textrm{1--Lip}}\sum_{w\in V}\phi(w)(\mu^{ p}_{x}(w)-\mu^{ p}_{y}(w))
\\
\nonumber
& = \sup_{\substack{\phi\in \textrm{1--Lip}\\ \phi:V\rightarrow \mathbb{Z}\\ \phi(y)=0}} \sum_{w\in V}\phi(w)(\mu^{ p}_{x}(w)-\mu^{ p}_{y}(w))
\\
\nonumber
& = \sup_{\substack{\phi\in \textrm{1--Lip}\\ \phi:V\rightarrow \mathbb{Z}\\ \phi(y)=0}}\left\{ \phi(x)\left(p-\frac{1-p}{d_{y}}\right)+\frac{1-p}{d_{x}}\sum_{\substack{w\sim x\\ w\neq y}}\phi(w) -\frac{1-p}{d_{y}}\sum_{\substack{w\sim y\\ w\neq x}}\phi(w)\right\}
\\
\nonumber
& = \sup_{\substack{\phi\in \textrm{1--Lip}\\ \phi:V\rightarrow \mathbb{Z}\\ \phi(y)=0}}\left\{ \phi(x)\left(p-\frac{1-p}{d_{y}}\right)+\frac{1-p}{d_{x}d_{y}}\left(d_{y}\sum_{\substack{w\sim x\\ w\neq y}}\phi(w)-d_{x}\sum_{\substack{w\sim y\\ w\neq x}}\phi(w)\right)\right\}
\\
\nonumber
& = \sup_{\substack{\phi\in \textrm{1--Lip}\\ \phi:V\rightarrow \mathbb{Z}\\ \phi(y)=0}}\left\{ \phi(x)\left(p-\frac{1-p}{d_{y}}\right)+\frac{1-p}{d_{x}d_{y}}F(\phi)\right\}
\\
\nonumber
& = \max_{j \in \{-1,0,1\}}\sup_{\phi \in \mathcal{A}_{j}}\left\{ j\left(p-\frac{1-p}{d_{y}}\right)+\frac{1-p}{d_{x}d_{y}}F(\phi)\right\}
\\
\nonumber
& = \max_{j \in \{-1,0,1\}}\left\{ j\left(p-\frac{1-p}{d_{y}}\right)+\frac{1-p}{d_{x}d_{y}}\sup_{\phi \in \mathcal{A}_{j}}F(\phi)\right\}
\\
\nonumber
& = \max_{j \in \{-1,0,1\}}\left\{ j\left(p-\frac{1-p}{d_{y}}\right)+\frac{1-p}{d_{x}d_{y}}c_j\right\}
\\
\label{eq:max}
& = \max\{f_{-1}( p), f_{0}( p), f_{1}( p)\}.
\end{align}
Therefore
$$\kappa_{p}(x,y) = 1 - \max\{f_{-1}(p), f_{0}(p), f_{1}(p)\}.$$
Since $\max\{f_{-1}( p), f_{0}( p), f_{1}( p)\}$ is the maximum of three linear functions of $p$, it is convex and piecewise linear in $p$ with at most $3$ linear parts, thus completing the proof.
\end{proof}

\section{Length of the last linear part}\label{last}
Before discussing the size of the last linear part we first need the
following lemma about some of the assumptions we can impose on an optimal
transport plan. We then show that, if different idlenesses $p_1 < p_2$
share a joint optimal Kantorovich potential, then the Ollivier-Rici
idleness function is linear on the whole interval $[p_1,p_2]$. This
was already mentioned in \cite{Smith14} for the special case $p_2=1$.

\begin{lemma}\label{goodplanlemma}
Let $\mu_{1}$ and $\mu_{2}$ be probability measures on $V.$ Then there exists an optimal transport plan $\pi$ transporting $\mu_{1}$ to $\mu_{2}$ with the following property:
For all $x\in V$ with $\mu_{1}(x)\leq \mu_{2}(x)$ we have $\pi(x,x) = \mu_{1}(x).$
\end{lemma}

This Lemma could be proved using Corollary 1.16 in
\cite{Vill03} (Invariance of
Kantorovich-Rubinstein distance under mass subtraction) but we present a proof in our much simpler context, for
the reader's convenience.

\begin{proof}
Let $\pi$ be an optimal transport plan transporting $\mu_{1}$ to $\mu_{2}$. Assume there exists an $x\in V$ with $\mu_{1}(x)\leq \mu_{2}(x)$, but $\pi(x,x)<\mu_{1}(x)$. Let $\mathcal{I} = \{z\in V\setminus \{x\} : \pi(z,x)>0\}$ and $\mathcal{O} = \{w\in V\setminus \{x\} : \pi(x,w)>0\}.$
Since $\pi$ is optimal, we must have $\mathcal{I}\cap\mathcal{O} = \emptyset.$ Then the relavent part of $\pi$ can be depicted as
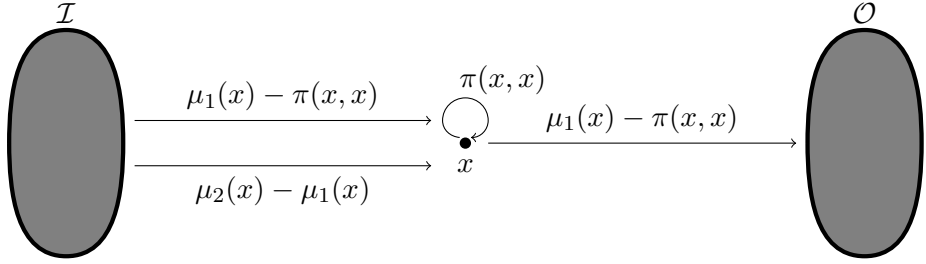
\begin{center}
\begin{tikzpicture}
[x=1.5cm, y=1.5cm,
	vertex/.style={
		shape=circle, fill=black, inner sep=1.5pt	
	}
]
\draw[ultra thick, fill = gray] (0,0) to [out=90,in=180] (0.5,1) to [out=0,in=90] (1,0) to [out=270,in=0] (0.5,-1) to [out=180,in=270] (0,0);

\draw[ultra thick, fill = gray] (7,0) to [out=90,in=180] (7.5,1) to [out=0,in=90] (8,0) to [out=270,in=0] (7.5,-1) to [out=180,in=270] (7,0);

\node[vertex, label=below:$x$] (x) at (4, 0) {};

\draw[->] (1.1,0.2) -- node[above]{$\mu_{1}(x)-\pi(x,x)$} (3.7, 0.2);
\draw[->] (1.1,-0.2) -- node[below]{$\mu_{2}(x)-\mu_{1}(x)$} (3.7, -0.2);
\draw[->] (4.2,0) -- node[above]{$\mu_{1}(x)-\pi(x,x)$} (6.9, 0);

\draw [->] (3.95,0.05) to [out=175,in=270] (3.8,0.2)
to [out=90,in=180] node[above right]{$\pi(x,x)$}(4,0.4) to [out=0,in=90] (4.2,0.2) to [out=270,in=0] (4.05,0.05) ;

\node at (0.5,1.15) {$\mathcal{I}$};
\node at (7.5,1.15) {$\mathcal{O}$};
\end{tikzpicture}
\end{center}
We now modify $\pi$ to obtain a new transport plan $\pi'$ as follows:
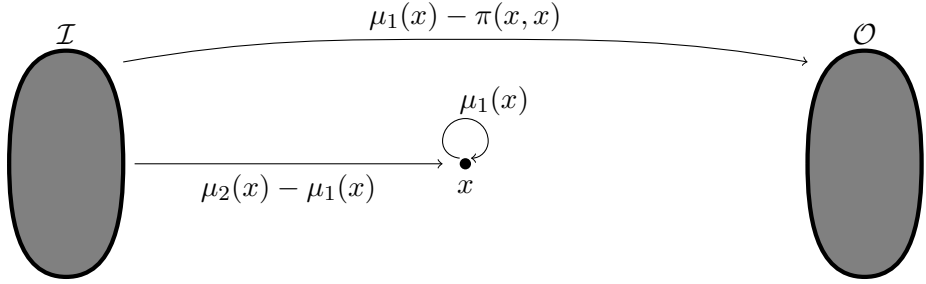
\begin{center}
\begin{tikzpicture}
[x=1.5cm, y=1.5cm,
	vertex/.style={
		shape=circle, fill=black, inner sep=1.5pt	
	}
]
\draw[ultra thick, fill = gray] (0,0) to [out=90,in=180] (0.5,1) to [out=0,in=90] (1,0) to [out=270,in=0] (0.5,-1) to [out=180,in=270] (0,0);

\draw[ultra thick, fill = gray] (7,0) to [out=90,in=180] (7.5,1) to [out=0,in=90] (8,0) to [out=270,in=0] (7.5,-1) to [out=180,in=270] (7,0);

\node[vertex, label=below:$x$] (x) at (4, 0) {};

\draw[->] (1.1,0) -- node[below]{$\mu_{2}(x)-\mu_{1}(x)$} (3.8, 0);

\draw[->] (1,0.9) to [out=10,in=180] (4,1.1) to [out=0,in=170](7,0.9);

\draw [->] (3.95,0.05) to [out=175,in=270] (3.8,0.2)
to [out=90,in=180] node[above right]{$\mu_1(x)$}(4,0.4) to [out=0,in=90] (4.2,0.2) to [out=270,in=0] (4.05,0.05) ;

\node at (0.5,1.15) {$\mathcal{I}$};
\node at (7.5,1.15) {$\mathcal{O}$};
\node at (4,1.3) {$\mu_{1}(x)-\pi(x,x)$};
\end{tikzpicture}
\end{center}
This new transport plan $\pi'$ is still optimal (by the triangle inequality). Note that
$$\pi'(z,z) = \begin{cases}
\pi(z,z) & \text{if $z\neq x$,}
\\
\mu_{1}(x) & \text{if $z = x$.}
\end{cases}$$
Repeating this modification at all other vertices that violate the condition of the lemma successively gives us our required optimal transport plan.
\end{proof}

\begin{lemma}\label{ineqs}
Let $G= (V,E)$ be a locally finite graph. Let $x,y\in V$ with $x\sim y.$ Let $0\leq p_{1}\leq p_{2} \leq 1.$ If there exists a \textrm{\rm{1}--}Lipschitz function $\phi$ which is an optimal Kantorovich potential transporting $\mu_{x}^{p_{1}}$ to $\mu_{y}^{p_{1}}$ and transporting $\mu_{x}^{p_{2}}$ to $\mu_{y}^{p_{2}},$ then $W_{xy}:[0,1] \to \mathbb{R}$, 
$W_{xy}(p)=W_{1}(\mu_{x}^{p},\mu_{y}^{p})$, is linear on $[p_{1},p_{2}].$
\end{lemma}
\begin{proof}
Let $\alpha \in [0,1].$ The convexity of $W_{xy}$, see Remark \ref{concavity}, implies that
$$\alpha W_{xy}(p_{1})+(1-\alpha)W_{xy}(p_{2})\geq W_{xy}(\alpha p_{1}+ (1-\alpha)p_{2}).$$
It only remains to show the above inequality is in fact an equality. Observe that
\begin{align*}
\mu_{x}^{\alpha p_{1}+ (1-\alpha)p_{2}} & = \alpha \mu_{x}^{p_{1}}+(1-\alpha)\mu_{x}^{p_{2}},
\\
\mu_{y}^{\alpha p_{1}+ (1-\alpha)p_{2}} & = \alpha \mu_{y}^{p_{1}}+(1-\alpha)\mu_{y}^{p_{2}}.
\end{align*}
Then, setting $p=\alpha p_{1}+ (1-\alpha)p_{2}$, we have
\begin{align*}
W_{xy}(p) \geq & \sum_{w\in V} \phi(w)(\mu_{x}^{\alpha p_{1}+ (1-\alpha)p_{2}}(w)-\mu_{y}^{\alpha p_{1}+ (1-\alpha)p_{2}}(w))
\\
= & \sum_{w\in V}\phi(w)\left(\alpha \mu_{x}^{p_{1}}(w)+(1-\alpha)\mu_{x}^{p_{2}}(w) - \alpha \mu_{y}^{p_{1}}(w)-(1-\alpha)\mu_{y}^{p_{2}}(w)\right)
\\
= & \:\alpha\sum_{w\in V}\phi(w)\left(\mu_{x}^{p_{1}}(w)-\mu_{y}^{p_{1}}(w)\right)+ (1-\alpha)\sum_{w\in V}\phi(w)\left(\mu_{x}^{p_{2}}(w)-\mu_{y}^{p_{2}}(w)\right)
\\
= & \:\alpha W_{xy}(p_{1})+(1-\alpha)W_{xy}(p_{2}).
\end{align*}
\end{proof}

\begin{lemma}\label{phisharp}
Let $G= (V,E)$ be a locally finite graph. Let $x,y\in V$ with $x\sim y$ and $d_{x}\geq d_{y}.$ Let $ p\in \left(\frac{1}{1+d_{x}},1\right].$ Let $\phi$ be an optimal Kantorovich potential transporting $\mu^{ p}_{x}$ to $\mu^{ p}_{y}$. Then
$$\phi(x)-\phi(y) = 1.$$
\end{lemma}

\begin{proof}
Let $\pi$ be an optimal transport plan transporting $\mu^{ p}_{x}$ to $\mu^{ p}_{y}.$ We may assume that $\pi$ satisfies the conditions of Lemma \ref{goodplanlemma}. Since $p > \frac{1}{1+d_x}$, then $\mu_{x}^{ p}(y) = \frac{1-p}{d_{x}}< \frac{d_{x}p}{d_{x}} = p =\mu_{y}^{ p}(y)$, therefore there exists $z\in B_{1}(x) \setminus \{y\}$ such that $\pi(z,y)>0.$ If $z = x$ then $\phi(x)-\phi(y) = 1$, by Lemma \ref{mass-sharpness}. Suppose $z \sim y$ and $z\neq x$. Then observe that $\mu^{ p}_{x}(z) = \frac{1-p}{d_{x}} \le \frac{1-p}{d_{y}} = \mu^{ p}_{y}(z).$ Thus $\pi(z,y)= 0$, by Lemma \ref{goodplanlemma}, which contradicts our assumption that $\pi(z,y)>0$.

The only case left to consider is $z \sim x$, $z \nsim y$, $z \neq y$. Then $d(z,y) = 2$, in which case we have $\phi(z)-\phi(y) = 2$, by Lemma \ref{mass-sharpness}. Then
\begin{align*}
2 & = \phi(z) - \phi(y)
\\
& = \phi(z) - \phi(x) + \phi(x) - \phi(y)
\\
& \leq 1 + \phi(x) - \phi(y)
\\
& \leq 2,
\end{align*}
which implies $\phi(x) - \phi(y) = 1.$
\end{proof}
We are now ready to prove the main theorem of this section.

\begin{theorem}\label{lastpiece}
Let $G=(V,E)$ be a locally finite graph and let $x,y\in V$ with $x \sim y$ and $d_{x}\geq d_{y}$. Then $p \mapsto \kappa_{p}(x,y)$ is linear over $\left[\frac{1}{d_{x}+1},1\right].$
\end{theorem}

\begin{proof}
Let $1 > p_0 > \frac{1}{d_{x}+1}$ and $\phi$ be an optimal Kantorovich potential transporting $\mu^{p_0}_{x}$ to $\mu^{ p_0}_{y}$. Then, by Lemma \ref{phisharp}, we have $\phi(x)-\phi(y) = 1.$ Note that any 1--Lipschitz $\psi$ satisfying $\psi(x)-\psi(y)$ is an optimal Kantorovich potential transporting $\mu_{x}^{1}$ to $\mu_{y}^{1}.$ Thus, by Lemma \ref{ineqs}, $p \mapsto \kappa_{p}(x,y)$ is linear over
$[p_0, 1]$. By continuity of $p \mapsto \kappa_{p}(x,y)$, this linearity extends to $\left[\frac{1}{d_{x}+1}, 1\right].$
\end{proof}

\begin{rem}\label{trick}
  Note that the above proof shows the existence a $1$--Lipschitz function
  $\phi$ with $\phi(x)-\phi(y)=1$, which is an optimal Kantorovich potential
  for all $p \in \left[ \frac{1}{d_x+1},1 \right]$:
  We choose $\phi$ to be an optimal Kantorovich potential transporting
  $\mu^{p_0}_{x}$ to $\mu^{ p_0}_{y}$ for some
  $1 > p_0 > \frac{1}{d_{x}+1}$ and satisfying $\phi(x)-\phi(y)=1$, as
  in the proof of Theorem \ref{lastpiece}. Then both $W_{xy}$ and the function
  $$ p \mapsto \sum_{w \in V} \phi(w)(\mu_x^p(w)-\mu_y^p(w)) $$
  are linear over $[\frac{1}{d_x+1},1]$ and agree at $p = p_0$ and
  $p = 1$.  Therefore, they agree on the whole interval and,
  consequently, $\phi$ is an optimal Kantorovich potential for all
  $p \in \left[ \frac{1}{d_x+1},1 \right]$.
\end{rem}

\section{Length of the first linear part}\label{first}

\begin{lemma}\label{halvingworks}
Let $G=(V,E)$ be a locally finite graph.   Let $F$ be as defined in equation \eqref{eq:F}. 
 Then
  $$\sup_{\substack{\phi\in \textrm{\rm{1}--{\rm Lip}}\\\phi : V\rightarrow \mathbb{Z}\\\phi(x) = \phi(y) = 0}}F(\phi) = \sup_{\substack{\phi\in \textrm{\rm{1}--{\rm Lip}}\\\phi : V\rightarrow \mathbb{Z}/2\\\phi(x) = \phi(y) = 0}}F(\phi).$$
\end{lemma}

\begin{proof}
Pick $\phi_{0}\in \textrm{1--Lip}$ such that $\phi_{0} : V\rightarrow \mathbb{Z}/2, \phi_{0}(x) = \phi_{0}(y) = 0$ and
$$F(\phi_{0}) = \sup_{\substack{\phi\in \textrm{1--Lip}\\\phi : V\rightarrow \mathbb{Z}/2\\\phi(x) = \phi(y) = 0}}F(\phi).$$
Note that
$$\phi_{0}(v) = \frac{\floor{\phi_{0}(v)}+\ceil{\phi_{0}(v)}}{2},$$
for all $v\in V.$ Thus
$$F(\phi_{0}) = \frac{F(\floor{\phi_{0}})+F(\ceil{\phi_{0}})}{2}.$$
By combining this with $F(\phi_{0}) \geq F(\floor{\phi_{0}})$ and $F(\phi_{0}) \geq F(\ceil{\phi_{0}})$ we obtain $$F(\phi_{0}) = F(\floor{\phi_{0}}) = F(\ceil{\phi_{0}}).$$
Since $\floor{\phi_{0}}: V\rightarrow \mathbb{Z}$ this completes the proof.
\end{proof}

The rest of this section is devoted to the proof of the following result.

\begin{theorem}
Let $G=(V,E)$ be a locally finite graph. Let $x,y\in V$ with $x\sim y$ and $d_{x}\geq d_{y}.$ Let $\ell={\rm{lcm}}(d_{x},d_{y}).$ Then $p \mapsto \kappa_{p}(x,y)$ is linear over $\left[0,\frac{1}{\ell+1}\right].$
\end{theorem}

\begin{proof}
  Let $F,\mathcal{A}_{j}, c_{j}, f_{j}$ be as defined in the proof of
  Theorem \ref{linear}. In order to bound the length of the first linear part of
  $\kappa_{p}$, we look at the intersection points of the functions
  $f_{j}.$ First we derive inequalities between the constants $c_{j}.$ Note that
$$f_{j}\left(\frac{1}{d_{y}+1}\right) = \frac{1}{(d_{y}+1)d_{x}}c_{j},$$
for $j\in\{-1,0,1\}$. We claim that $ f_1\left(\frac{1}{d_y+1}\right) \ge f_j\left(\frac{1}{d_{y}+1}\right).$ It then follows that $c_{1}\geq c_{0}$ and $c_{1}\geq c_{-1}$. We now prove the claim:
\\
Note that $\frac{1}{d_y+1} \in \left[ \frac{1}{d_x+1},1 \right]$ and that,
by Remark \ref{trick}, there exists an optimal Kantorovich potential
$\phi$ at idleness $\frac{1}{d_y+1}$ with $\phi(x)-\phi(y)=1$. Therefore, by equation \eqref{eq:max},
$$ f_1\left(\frac{1}{d_y+1}\right) = W_{xy}\left(\frac{1}{d_y+1}\right) = \max\left\{ f_{-1}\left(\frac{1}{d_y+1}\right), f_0\left(\frac{1}{d_y+1}\right), f_1\left(\frac{1}{d_y+1}\right) \right\}, $$
which proves the claim.

Let $\phi_{j}\in\mathcal{A}_{j}$ satisfy  $F(\phi_{j})=c_j=\max_{\mathcal{A}_{j}} F$. Let $\psi = \frac{\phi_{-1}+\phi_{1}}{2}$. Note
that $\psi$ is 1--Lipschitz and
$\psi(x)=\psi(y)=0$. The
function $\psi$ may fail to be integer-valued
but we note that $\psi: V \rightarrow \mathbb{Z}/2$ and so, by Lemma \ref{halvingworks}, we have
\begin{equation}\label{eq:c-1c0c1}
c_{0}\geq F\left(\frac{\phi_{-1}+\phi_{1}}{2}\right)= \frac{c_{-1}+c_{1}}{2} \geq c_{-1}.
\end{equation}
Therefore $$c_{1}\geq c_{0} \geq c_{-1}.$$

Let $g = {\rm gcd}(d_{x},d_{y}).$ Since the constants $c_{j}$ are
integer linear combinations of $d_{x}$ and $d_{y}$, we have $g| c_{j}$
for $j\in\{-1,0,1\}.$ For the computation of the possible intersection points of $f_{j}$, we will make use of the following simple observation. Let $b>0$ and suppose that
$0\leq \frac{a}{a+b}\leq 1.$ Then $a>0.$

Suppose that $p'$ satisfies $f_{-1}( p') = f_{0}( p').$ Then
$$ p' = \frac{d_{x}-(c_{0}-c_{-1})}{d_{x}d_{y}+d_{x}-(c_{0}-c_{-1})}.$$
We can write $ c_{0}-c_{-1} = d_{x}-Kg$ for some $K\in\mathbb{Z}$. Then
$$ p' = \frac{Kg}{d_{x}d_{y}+Kg} = \frac{K}{\ell+K}, $$
with $\ell = {\rm{lcm}}(d_x,d_y)$.
Since $0\leq p' \leq 1$ we have $K\geq 0.$ Thus the smallest strictly positive intersection point is $ p'=\frac{1}{\ell+1}$. 

Now suppose that $p'$ satisfies $f_{1}( p') = f_{0}( p').$ Then
$$ p' = \frac{d_{x}-(c_{1}-c_{0})}{d_{x}d_{y}+d_{x}-(c_{1}-c_{0})}.$$
We can write $ c_{1} - c_{0} = d_{x}-Kg$ for some $K\in\mathbb{Z}$. Then
$$
 p' = \frac{Kg}{d_{x}d_{y}+Kg} = \frac{K}{\ell+K}.
$$
Since $0\leq p' \leq 1$ we have $K\geq 0.$ Thus the smallest strictly positive intersection point is again $ p'=\frac{1}{\ell+1}.$

Now suppose that $p'$ satisfies $f_{-1}( p') = f_{1}( p').$ Then 
\[
f_{-1}(p') =\frac{1}{2} ( f_{-1}(p')+f_{1}(p') ) = \frac{1-p'}{d_x d_y} \, \frac{c_{-1}+c_1}{2} \stackrel{\eqref{eq:c-1c0c1}}{\le} 
\frac{1-p'}{d_x d_y} \, c_0 = f_0(p').
\]
In particular 
$$ f_1(p') = f_{-1}(p') =\frac{1}{2} ( f_{-1}(p')+f_{1}(p') ) \leq f_0(p'). $$
Thus either $f_{0}(p')>f_{-1}(p')$ and $f_{0}(p')>f_{1}(p')$, in which case there is no turning point at $p'$, or $f_{0}(p') = f_{-1}(p') = f_{1}(p')$, in which case $p'$ is one of the points we have already considered. Thus $p\mapsto\kappa_{p}(x,y)$ is linear over $[0,\frac{1}{\ell+1}].$
\end{proof}

Let us finish this section with some observations about relations
between various different curvature values.
 Assume that  $d_y | d_x$. Then lcm$(d_x,d_y) = \max(d_x,d_y)$ and so, by Theorem \ref{main},
 $p \mapsto \kappa_{p}(x,y)$ has at most two linear parts. We can give a formula for $\kappa_{p}(x,y)$ in terms of the curvatures $\kappa_0(x,y)$ and $\kappa(x,y)$ by using the fact that $\kappa_{p}(x,y)$ can change its slope only at $p=\frac{1}{d_x+1}$ and that $\kappa_1=0$, $\kappa_1'=-\kappa$. This formula, given in the following theorem, emerges via a
straightforward calculation and applies, in particular, to all regular
graphs.

\begin{theorem}\label{formula}
Let $G = (V,E)$ be a locally finite graph. Let $x,y\in V$ with $x\sim y$ and $d_{y}| d_{x}.$ Then
$$\kappa_{ p}(x,y) = \begin{cases} (d_{x}\kappa(x,y)-(d_{x}+1)\kappa_{0}(x,y)) p + \kappa_{0}(x,y), &   \text{if $p\in [0,\frac{1}{d_{x}+1}]$,}\\
(1- p)\kappa(x,y), &   \text{if $p \in [\frac{1}{d_{x}+1},1]$.}
\end{cases} $$
\end{theorem}

\begin{rem}
  As mentioned earlier $\kappa_{\frac{1}{2}},\kappa_{\frac{1}{d+1}}$
  and $\kappa$ have been studied in various articles. In fact, the identity
  \begin{equation}
  \label{eq:**}
   \kappa_p(x,y) = (1-p)\kappa(x,y) 
  \end{equation}
  holds true at all edges (even those whose Ollivier-Ricci idleness
  function has three linear parts) and for all values
  $p \in \left [\frac{1}{\max\{d_x,d_y\}+1},1 \right]$. Equation \eqref{eq:**} follows from Theorem
  \ref{lastpiece} and the fact that $\kappa_1=0$, $\kappa_1'=-\kappa$. As a consequence, we have
  $$\kappa = 2\kappa_{\frac{1}{2}} = \frac{d+1}{d}\kappa_{\frac{1}{d+1}}.$$
\end{rem}

We end this section with a connection between $\kappa$ and $\kappa_{0}.$

\begin{theorem}
Let $G = (V,E)$ be a locally finite graph. Let $x,y \in V$ with $x\sim y$ and $d_{x}\geq d_{y}.$ Then
$$\kappa_{0}(x,y)\leq \kappa(x,y)\leq \kappa_{0}(x,y)+\frac{2}{d_{x}}.$$
\end{theorem}

\begin{proof}
The first inequality follows from the fact that the graph of a concave function lies below its tangent line at each point and that
$\kappa_1=0$, $ \kappa = -\kappa_1'$:
\[
\kappa_0 \le \kappa_1 + \kappa'_1 (0-1) = \kappa.
\]
Now we prove the second inequality.
Let $\phi$ be a 1--Lipschitz function with $\phi(y)=0$ such that
$$W_{xy}(0) = \sum_{w \in V} \phi(w) (\mu_x^0(w) - \mu_y^0(w)) =
\frac{-1}{d_{y}}\phi(x)  + \frac{1}{d_{x}}\sum_{\substack{z\sim x\\ z\neq y}} \phi(z) - \frac{1}{d_{y}}\sum_{\substack{z\sim y \\ z\neq x}}\phi(z).$$
Then
\begin{align*}
W_{xy}\left(\frac{1}{d_{x}+1}\right)
 & \geq \sum_{w \in V} \phi(w) (\mu_x^{\frac{1}{d_x+1}}(w) - \mu_y^{\frac{1}{d_x+1}}(w))
\\
& = \left(\frac{1}{d_{x}+1}-\frac{d_{x}}{(d_{x}+1)d_{y}}\right)\phi(x)  + \frac{1}{d_{x}+1}\sum_{\substack{z\sim x \\ z\neq y}} \phi(z) - \frac{d_{x}}{(d_{x}+1)d_{y}}\sum_{\substack{z\sim y \\ z\neq x}}\phi(z).
\end{align*}
Thus
\begin{align*}
\frac{d_{x}+1}{d_{x}}W_{xy}\left(\frac{1}{d_{x}+1}\right) 
& \geq \left(\frac{1}{d_{x}}-\frac{1}{d_{y}}\right)\phi(x)  + \frac{1}{d_{x}}\sum_{\substack{z\sim x \\ z\neq y}} \phi(z) - \frac{1}{d_{y}}\sum_{\substack{z\sim y \\ z\neq x}}\phi(z) 
\\
& = W_{xy}(0) + \frac{1}{d_{x}} \phi(x)
\\
& = W_{xy}(0) + \frac{1}{d_{x}} (\phi(x)-\phi(y))
\\
& \geq W_{xy}(0) - \frac{1}{d_{x}}
\end{align*}
since $\phi$ is 1--Lipschitz.
Therefore
\begin{align*}
\kappa_{\frac{1}{d_{x}+1}}(x,y) & \leq 1+ \frac{1}{d_{x}+1} -\frac{d_{x}}{d_{x}+1}W_{xy}(0)
\\
& = \frac{2}{d_{x}+1} + \frac{d_{x}}{d_{x}+1}(1-W_{xy}(0))
\\
& = \frac{2}{d_{x}+1} + \frac{d_{x}}{d_{x}+1}\kappa_{0}(x,y).
\end{align*}
Finally, by \eqref{eq:**},
$$\kappa(x,y) = \frac{d_{x}+1}{d_{x}}\kappa_{\frac{1}{d_{x}+1}}(x,y) \leq \kappa_{0}(x,y) + \frac{2}{d_{x}}.$$
\end{proof}

\begin{rem}
Let $G = (V,E)$ be a locally finite $d$-regular graph. Let $x,y \in V$ with $x\sim y.$ Then by the above theorem we have
$$\kappa_{0}(x,y)\leq \kappa(x,y)\leq \kappa_{0}(x,y)+\frac{2}{d}.$$
Furthermore, by \cite{LoRo}, $\kappa_{0}(x,y)\in\mathbb{Z}/d.$ Similar arguments show that $\kappa(x,y)\in\mathbb{Z}/d.$ Thus
$$\kappa(x,y) = \kappa_{0}(x,y)+\frac{C}{d},$$
where $C\in\{0,1,2\}.$
\end{rem}

%





\section{Application to the Cartesian product}\label{CartSection}

In \cite{LLY11} the authors proved the following results on the curvature of Cartesian products of graphs:
\begin{theorem}[\cite{LLY11}]
Let  $G=(V_{G},E_{G})$ be a $d_{G}$-regular graph and $H=(V_{H},E_{H})$ be a $d_{H}$-regular graph. Let $x_{1},x_{2}\in V_{G}$ with $x_{1}\sim x_{2}$ and $y\in V_{H}$. Then
\begin{align*}
\kappa^{G\times H}((x_{1},y),(x_{2},y)) & = \frac{d_{G}}{d_{G}+d_{H}} \kappa^{G}(x_{1},x_{2}),
\\
\kappa^{G\times H}_{0}((x_{1},y),(x_{2},y)) & = \frac{d_{G}}{d_{G}+d_{H}} \kappa^{G}_{0}(x_{1},x_{2}).
\end{align*}
\end{theorem}
Using our formula from Theorem \ref{formula}, we extend this result and derive
relations between the full Ollivier-Ricci idleness functions involved in
the Cartesian product.
\\
\\
{\bf Corollary \ref{cartcor}.}
{\it
Let $G=(V_{G},E_{G})$ be a $d_{G}$-regular graph and $H=(V_{H},E_{H})$ be a $d_{H}$-regular graph. Let $x_{1},x_{2}\in V_{G}$ with $x_{1}\sim x_{2}$ and $y\in V_{H}$. Then
\begin{align*}
& \kappa^{G\times H}_{ p}((x_{1},y),(x_{2},y))
\\
& = \begin{cases} \frac{d_{G}}{d_{G}+d_{H}}\kappa^{G}_{ p}(x_{1},x_{2})+\frac{d_{G}d_{H}}{d_{G}+d_{H}}(\kappa^{G}(x_{1},x_{2})-\kappa^{G}_{0}(x_{1},x_{2})) p,&  \text{if $p \in [0,\frac{1}{d_{G}+d_{H}+1}]$,}\\
\frac{d_{G}}{d_{G}+d_{H}} \kappa^{G}(x_{1},x_{2})(1- p), &  \text{if $p\in [\frac{1}{d_{G}+d_{H}+1},1]$.}
\end{cases}
\end{align*}
}

\begin{proof}
For ease of reading, we define $\kappa^{G\times H}_{ p}:= \kappa^{G\times H}_{ p}((x_{1},y),(x_{2},y))$ and $\kappa^{G}_{ p}:=\kappa^{G}_{ p}(x_{1},x_{2})$.
Let $ p \in [0,\frac{1}{d_{G}+d_{H}+1}].$ Then, by Theorem \ref{formula},
\begin{align*}
\kappa^{G\times H}_{ p} & = ((d_{G}+d_{H})\kappa^{G\times H}-(d_{G}+d_{H}+1)\kappa^{G\times H}_{0}) p+\kappa^{G\times H}_{0}
\\
& =\frac{d_{G}}{d_{G}+d_{H}}\left\{((d_{G}+d_{H})\kappa^{G}-(d_{G}+d_{H}+1)\kappa^{G}_{0}) p+\kappa^{G}_{0}\right\}
\\
& =\frac{d_{G}}{d_{G}+d_{H}}\left\{(d_{G}\kappa^{G}-(d_{G}+1)\kappa^{G}_{0}) p+\kappa^{G}_{0}\right\}+\frac{d_{G}d_{H}}{d_{G}+d_{H}}(\kappa^{G}-\kappa^{G}_{0}) p
\\
& = \frac{d_{G}}{d_{G}+d_{H}}\kappa^{G}_{ p}+\frac{d_{G}d_{H}}{d_{G}+d_{H}}(\kappa^{G}-\kappa^{G}_{0}) p.
\end{align*}
Now suppose $ p \in [\frac{1}{d_{G}+d_{H}+1},1].$ Then
$$\kappa_{ p}^{G\times H}  = \kappa^{G\times H}(1- p) = \frac{d_{G}}{d_{G}+d_{H}} \kappa^{G}(1- p).$$
\end{proof}

\section{Bone idleness and some open questions}\label{QSection}
We finish this article with a discussion of when the Ollivier-Ricci
idleness function $p \mapsto \kappa_p(x,y)$ is globally linear for all
edges. First we introduce the notion \emph{bone idle}.
\begin{definition}
Let $G=(V,E)$ be a locally finite graph. We say an edge $x \sim y$ is \emph{bone idle} if $\kappa_{p}(x,y) = 0$ for every $p\in[0,1].$   We say that $G$ is \emph{bone idle} if every edge is bone idle. 
\end{definition}
\begin{rem}
\label{rmk:7.1}
Note that $\kappa_{p}(x,y) = 0$ for all $p\in[0,1]$ if and only if $\kappa_{0}(x,y) = \kappa(x,y) = 0.$ This follows from the concavity of $\kappa_{p}(x,y).$
\end{rem}
It is an interesting problem to classify the graphs which are bone
idle. Due to the above remark this question is closely related to
various notions of Ricci flatness. The following two results allows us
to classify bone idle graphs with girth at least $5$. Recall that
the girth of a graph is the length of its shortest non-trivial cycle.
\begin{theorem}[\cite{LLY13}]
\label{thm:7.2}
Let $G = (V,E)$ be a locally finite graph with girth at least $5.$ Suppose that $\kappa(x,y)=0$ for all $x,y\in V$ with $x\sim y.$ Then $G$ is isomorphic to one of the following graphs:
\begin{enumerate}[(i)]
\item
The infinite path;
\item
The cyclic graph $C_{n}$ for $n\geq 6$;
\item
The dodecaheral graph;
\item
The Petersen graph;
\item
The half-dodecahedral graph.
\end{enumerate}
\end{theorem}
\begin{theorem}[\cite{BM15}]
\label{thm:7.3}
Let $G = (V,E)$ be a locally finite graph with girth at least $5.$ Suppose that $\kappa_{0}(x,y)=0$ for all $x,y\in V$ with $x\sim y.$ Then $G$ is isomorphic to one of the following graphs:
\begin{enumerate}[(i)]
\item
The infinite path;
\item
The cyclic graph $C_{n}$ for $n\geq 6$;
\item
The path $P_{n}$ for $n\geq 2$;
\item
The Star graph $S_{n}$ for $n\geq 3$.
\end{enumerate}
\end{theorem}

Combining Theorems \ref{thm:7.2}, \ref{thm:7.3} and Remark \ref{rmk:7.1} gives the following:
\begin{corollary}
Let $G = (V,E)$ be a locally finite graph with girth at least $5.$ Suppose that $G$ is bone idle. Then $G$ is isomorphic to one of the following graphs:
\begin{enumerate}[(i)]
\item
The infinite path;
\item
The cyclic graph $C_{n}$ for $n\geq 6$.
\end{enumerate}
\end{corollary}
\begin{rem}
Note that the above Corollary shows that there exists no bone idle graph with girth equal to $5.$
\end{rem}
The full classification of bone idle graphs is still open.

The condition of bone idleness of an edge $x \sim y$ can be weakened to
only require that $p \mapsto \kappa_p(x,y)$ is globally linear on
$[0,1]$. This is equivalent to $\kappa_{0}(x,y) = \kappa(x,y).$
It is a natural desire to understand this weaker condition better.

Recall that $\kappa\geq \kappa_{0}.$ The Petersen graph has
$\kappa(x,y) = 0$ and $\kappa_{0}(x,y)< 0$ for all edges $x \sim y$. It
is thus a natural question to ask whether there exists a graph $G$
with an edge $x \sim y$ satisfying $\kappa(x,y)>0$ and
$\kappa_{0}(x,y)< 0$. We do not know of any such example.
\\
\\
{\bf Acknowledgements}
\\
DC and NP thank UTSC (Hefei) and TSIMF (Sanya) for their hospitality where a lot of the above work was carried out. FM wants to thank the German National Merit Foundation for financial support. DC wants to thank the EPSRC for financial support through his postdoctoral prize. Thanks to George Stagg (Newcastle University) for his continued computing assistance.

\end{document}